\documentclass[12pt]{article}
\usepackage{amssymb}
\usepackage{amsmath}
\usepackage{amsthm}
\usepackage{color}
\usepackage{comment}
\usepackage{geometry}		
\usepackage{graphicx}

\DeclareFontFamily{OT1}{pzc}{}
\DeclareFontShape{OT1}{pzc}{m}{it}{<-> s * [1.10] pzcmi7t}{}
\DeclareMathAlphabet{\mathpzc}{OT1}{pzc}{m}{it}

\let\originalleft\left
\let\originalright\right
\renewcommand{\left}{\mathopen{}\mathclose\bgroup\originalleft}
\renewcommand{\right}{\aftergroup\egroup\originalright}

\begin{document}

\newcommand\bO{{\bf 0}}
\newcommand\ee{\varepsilon}
\newcommand\co{\mathpzc{o}}
\newcommand\cF{\mathcal{F}}
\newcommand\cO{\mathcal{O}}
\newcommand\re{{\rm e}}

\newcommand{\myStep}[2]{{\bf Step #1} --- #2\\}

\newtheorem{theorem}{Theorem}[section]
\newtheorem{corollary}[theorem]{Corollary}
\newtheorem{lemma}[theorem]{Lemma}
\newtheorem{proposition}[theorem]{Proposition}
\newtheorem{conjecture}[theorem]{Conjecture}

\theoremstyle{definition}
\newtheorem{definition}{Definition}[section]
\newtheorem{example}[definition]{Example}

\theoremstyle{remark}
\newtheorem{remark}{Remark}[section]




\title{
On the stability of boundary equilibria in Filippov systems.
}
\author{
D.J.W.~Simpson\\\\
School of Fundamental Sciences\\
Massey University\\
Palmerston North\\
New Zealand
}
\maketitle


\begin{abstract}

The leading-order approximation to a Filippov system $f$ about a generic boundary equilibrium $x^*$
is a system $F$ that is affine one side of the boundary and constant on the other side.
We prove $x^*$ is exponentially stable for $f$ if and only if it is exponentially stable for $F$
when the constant component of $F$ is not tangent to the boundary.
We then show exponential stability and asymptotic stability are in fact equivalent for $F$.
We also show exponential stability is preserved under small perturbations to the pieces of $F$.
Such results are well known for homogeneous systems.
To prove the results here additional techniques are required because the two components of $F$ have different degrees of homogeneity.
The primary function of the results is to reduce the problem of the stability of $x^*$ from
the general Filippov system $f$ to the simpler system $F$.
Yet in general this problem remains difficult.
We provide a four-dimensional example of $F$
for which orbits appear to converge to $x^*$ in a chaotic fashion.
By utilising the presence of both homogeneity and sliding motion
the dynamics of $F$ can in this case be reduced to the combination of a one-dimensional return map and a scalar function.

\end{abstract}

\section{Introduction}
\label{sec:intro}
\setcounter{equation}{0}

Filippov systems are piecewise-smooth vector fields
for which evolution on discontinuity surfaces (termed sliding motion) is permitted and defined
by appropriately averaging the neighbouring smooth components of the vector field.
Such systems provide useful mathematical models for a wide variety of physical phenomena
that switch between two or more modes of operation \cite{DiBu08}.
Sliding motion represents the idealised limit that the time between consecutive switching events is zero.

As the parameters of a Filippov system are varied in a continuous fashion,
a {\em regular equilibrium} (zero of a smooth component of the vector field) can collide with a discontinuity surface.
This is known as a {\em boundary equilibrium bifurcation}.
Generically one {\em pseudo-equilibrium} (zero of the vector field for sliding motion) emanates from the bifurcation.
At the bifurcation it is possible for a stable regular equilibrium to simply transition to a stable pseudo-equilibrium,
but other invariant sets can be created, such as limit cycles and even chaotic sets \cite{Gl18}.
Most critically, even if the regular equilibrium is stable (in fact exponentially stable)
it is possible that no attractor exists on the other side of the bifurcation (multiple attractors are also possible \cite{Si18d}).

In order to determine the dynamics near a boundary equilibrium bifurcation,
it is helpful to know whether or not the equilibrium at the bifurcation is stable.
In particular, if the boundary equilibrium is asymptotically stable
then a local attractor must exist on both sides of the bifurcation \cite{DiNo08}.

For two-dimensional systems this stability problem is straight-forward.
There are eight topologically distinct cases for a generic boundary equilibrium \cite{Fi88}.
In two of these cases the equilibrium is exponentially stable, while in the remaining six it is unstable.
In higher dimensions such a characterisation is likely to be unachievable.
In \cite{Si18d} a three-dimensional system was given for which
both the regular and pseudo-equilibria are asymptotically stable (for their associated smooth vector fields)
but the boundary equilibrium is unstable.

Arguably the most useful tool for showing that a boundary equilibrium is stable is a Lyapunov function \cite{Jo03}.
Lyapunov functions are widely employed in control theory,
although the presence of sliding motion adds some complexity to this approach \cite{DeRo14,IeVa15,MoPy89}.
Existing methods search for Lyapunov functions within some class,
so if a Lyapunov function is not found then no conclusion can be made about stability.

\begin{figure}[b!]
\begin{center}
\includegraphics[height=4.5cm]{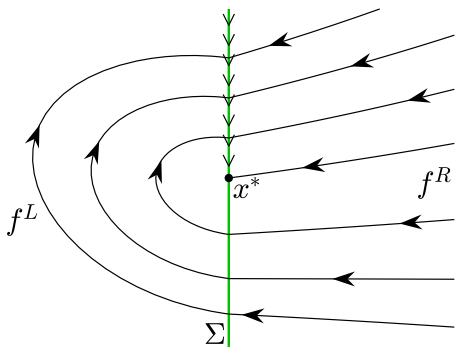}
\caption{
A phase portrait of a two-dimensional Filippov system with an exponentially stable boundary equilibrium $x^*$.
To the left [right] of the discontinuity surface $\Sigma$,
the dynamics is governed by $\dot{x} = f^L(x)$ [$\dot{x} = f^R(x)$].
The central point $x^* \in \Sigma$ is a zero of $f^L$ but not of $f^R$.
On $\Sigma$ orbits above $x^*$ slide towards $x^*$.
\label{fig:expStabSchem_a}
} 
\end{center}
\end{figure}

The results of this paper simplify the stability problem by justifying the removal of higher-order terms.
A smooth discontinuity surface $\Sigma$ of a Filippov system $f$
locally divides phase space into regions where two different smooth components apply, call them $f^L$ and $f^R$.
A boundary equilibrium is a point $x^* \in \Sigma$ that is a zero of one of these vector fields, say $f^L$,
as in Fig.~\ref{fig:expStabSchem_a}.
By replacing $f^L$ and $f^R$ with their leading-order approximations about $x^*$, we obtain a reduced system $F$.
Specifically $f^L$ is replaced by the affine vector field $D f^L(x^*) (x - x^*)$,
while $f^R$ is replaced by the constant vector field $f^R(x^*)$.
Below we prove the following three results.
First we show that $x^*$ is exponentially stable for $f$ if and only if it is exponentially stable for $F$
assuming $f^R(x^*)$ is not tangent to $\Sigma$ (Theorem \ref{th:expStabEquiv}).
Second we show that, although asymptotic stability is a weaker form of stability than exponential stability,
these two types of stability are in fact equivalent for $F$ (Theorem \ref{th:asyStabImpliesExpStab}).
Third we show that exponential stability of $x^*$ for $F$ is robust to small perturbations in the entries of $D f^L(x^*)$ and $f^R(x^*)$
(Theorem \ref{th:robust}).

These results are well-known when $x^*$ is instead a zero of both $f^L$ and $f^R$ (so both pieces of $F$ are affine)
and more generally for perturbations of homogeneous vector fields \cite{Fi88,LaSt71}.
Recall a vector field $g$ is {\em homogeneous of degree $d$} if
$g(\gamma x) = \gamma^d g(x)$ for all $\gamma > 0$ and $x \in \mathbb{R}^n$.
For our vector field $F$ with $x^*$ located at the origin,
one component of $F$ is homogeneous of degree $1$ while the other component is homogeneous of degree $0$.
For this reason additional arguments are required to obtain the results.

The remainder of this paper is organised as follows.
In \S\ref{sec:mainResult} we clarify notation, provide some basic definitions,
and precisely state the main results.
The next few sections prepare proofs of these results.
Since stability is defined in terms of Filippov solutions
we start in \S\ref{sec:euc} by recalling (from \cite{Fi88}) some fundamental theorems on the existence, uniqueness, and robustness of Filippov solutions.
Then in \S\ref{sec:sliding} we prove a lemma formally connecting sliding motion to the definition of a Filippov solution.

To prove Theorem \ref{th:expStabEquiv} we perform a spatial scaling to `blow-up' the origin.
For homogeneous systems this spatial scaling can be equated to a simple rescaling of time.
Indeed this is why at different distances from the origin solutions follow the same paths (just at different speeds).
The same behaviour occurs in our setting except, as shown in \S\ref{sec:lh}, the required time scaling is a discontinuous function of $x$.
Accommodating this discontinuity is the main technical hurdle that must be overcome
and this is achieved in \S\ref{sec:timeScaling} by applying tools from real analysis.
Then in \S\ref{sec:mainProof} we prove Theorems \ref{th:expStabEquiv}--\ref{th:robust}.

In order to highlight the complexity of the problem of the stability of $x^*$ for $F$,
in \S\ref{sec:chaos} we introduce a four-dimensional example
for which numerical simulations suggest Filippov solutions converge to the origin in a chaotic fashion.
This phenomenon does not appear to have been described before, possibly because it requires a relatively large number of dimensions (four),
but we believe it occurs generically, i.e.~our example is not special.
The system can be well understood because, rather remarkably,
the dynamics can be reduced from four dimensions down to only one dimension via the construction of a return map.
A return map (or Poincar\'e map) always foments the loss of one dimension.
For our example, Filippov solutions repeatedly undergo sliding motion.
Since this occurs on a codimension-one surface we are able to remove another dimension
(this technique is also employed in \cite{Gl18,GlJe15}).
By also utilising the (piecewise) homogeneity of $F$ we can remove a third dimension.

Finally in \S\ref{sec:conc} we provide a conjecture regarding Lyapunov stability
and summarise how the different types of stability for $f$ and $F$ imply one another.

\section{Preliminaries and main results}
\label{sec:mainResult}
\setcounter{equation}{0}

Here we first clarify some basic notation, \S\ref{sub:notation}.
Then in \S\ref{sub:FilippovGeneral} we define Filippov solutions and the stability of equilibria.
Then we consider an arbitrary boundary equilibrium of a Filippov system $f$
and construct the approximation $F$, \S\ref{sub:be}.
Lastly in \S\ref{sub:results} we state the main results, Theorems \ref{th:expStabEquiv}--\ref{th:robust}.

\subsection{Key notation}
\label{sub:notation}

We consider $\mathbb{R}^n$ ($n \ge 1$) with the $n$-dimensional Lebesgue measure and a norm $\| \cdot \|$.
The origin is denoted $\bO$.
Open and closed balls centred at $x \in \mathbb{R}^n$ with radius $\delta > 0$ are denoted
\begin{align*}
B_\delta(x) &= \left\{ y \in \mathbb{R}^n \,\big|\, \| x - y \| < \delta \right\}, \\
\overline{B}_\delta(x) &= \left\{ y \in \mathbb{R}^n \,\big|\, \| x - y \| \le \delta \right\}.
\end{align*}
A function $E : \mathbb{R}^n \to \mathbb{R}^n$ is
\begin{align*}
\text{$\co(x)$ if} \lim_{x \to \bO} \frac{\| E(x) \|}{\| x \|} &= 0, \\
\text{and $\cO(x)$ if} \limsup_{x \to \bO} \frac{\| E(x) \|}{\| x \|} &< \infty.
\end{align*}

\subsection{Filippov solutions and the stability of equilibria}
\label{sub:FilippovGeneral}

Let $\Omega \subset \mathbb{R}^n$ be open and connected
and let $\Sigma \subset \Omega$ be a set with zero measure.
Given a function $f : \Omega \setminus \Sigma \to \mathbb{R}^n$ that is measurable and bounded on any bounded subset of $\Omega \setminus \Sigma$,
we are interested in solutions to the system $\dot{x} = f(x)$.
To this end, following Filippov \cite{Fi88,Fi60}, we define
the set-valued function:
\begin{equation}
\cF(x) = \bigcap_{\delta > 0} {\rm co} \left[ f \left( B_\delta(x) \setminus \Sigma \right) \right],
\label{eq:cFgeneral}
\end{equation}
where ${\rm co}[U]$ denotes the smallest closed convex set containing $U \subset \mathbb{R}^n$.
For each $x \in \Omega$, $\cF(x)$ represents the smallest closed convex set containing all
limiting values $\lim_{i \to \infty} f(x_i)$, where $x_i \to x$ with $x_i \notin \Sigma$.

\begin{definition}
An absolutely continuous function\footnote{
A function $\phi$ is {\em absolutely continuous} on $[a,b]$
if for all $\ee > 0$ there exists $\delta > 0$ such that
$\sum_i \| \phi(b_i) - \phi(a_i) \| < \ee$ for any finite collection of pairwise disjoint intervals
$[a_i,b_i] \subset [a,b]$ satisfying $\sum_i (b_i - a_i) < \delta$.
Note that absolutely continuous functions are differentiable almost everywhere.
}
$\phi : [a,b] \to \Omega$ is a {\em Filippov solution}
to $\dot{x} = f(x)$ if $\dot{\phi}(t) \in \cF(\phi(t))$ for almost all $t \in [a,b]$.
\label{df:FilippovSolution}
\end{definition}

\begin{definition}
A point $x^* \in \Omega$ is an {\em equilibrium} of $\dot{x} = f(x)$
if $\phi(t) = x^*$ is a Filippov solution to $\dot{x} = f(x)$ for all $t \in \mathbb{R}$.
\end{definition}

\begin{definition}
An equilibrium $x^* \in \Omega$ of $\dot{x} = f(x)$ is said to be
\begin{enumerate}
\item
{\em Lyapunov stable} if for all $\ee > 0$ there exists $\delta > 0$ such that
every Filippov solution $\phi(t)$ with $\phi(0) \in B_\delta(x^*)$ has $\phi(t) \in B_\ee(x^*)$ for all $t \ge 0$;
\item
{\em asymptotically stable} if it is Lyapunov stable and there exists $\delta > 0$ such that
every Filippov solution $\phi(t)$ with $\phi(0) \in B_\delta(x^*)$ has $\phi(t) \to x^*$ as $t \to \infty$;
\item
{\em exponentially stable} if there exists $\alpha \ge 1$, $\beta > 0$, and $\delta > 0$ such that
every Filippov solution $\phi(t)$ with $\phi(0) \in B_\delta(x^*)$ has
$\| \phi(t) - x^* \| \le \alpha \re^{-\beta t} \| \phi(0) - x^* \|$ for all $t \ge 0$.
\end{enumerate}
\label{df:stable}
\end{definition}

\subsection{Boundary equilibria}
\label{sub:be}

The phase space of an $n$-dimensional piecewise-smooth system $\dot{x} = f(x)$
contains $(n-1)$-dimensional {\em discontinuity surfaces}
that divide the space into regions where $f$ is smooth.
In a neighbourhood of a point on exactly one discontinuity surface
only two smooth subsystems of $f$ are involved.
We can choose coordinates such that, at least locally, this discontinuity surface
is the coordinate plane $x_1 = 0$ (where $x_1$ is the first component of $x$).
Then, in this neighbourhood, the system takes the form
\begin{equation}
\dot{x} = f(x) = \begin{cases}
f^L(x), & x_1 < 0, \\
f^R(x), & x_1 > 0.
\end{cases}
\label{eq:f}
\end{equation}
Now let $\Omega \subset \mathbb{R}^n$ be an open, connected set and let
\begin{equation}
\Sigma = \left\{ x \in \Omega \,\big|\, x_1 = 0 \right\}
\nonumber
\end{equation}
be the part of the discontinuity surface that belongs to $\Omega$.
Also let
\begin{equation}
\begin{split}
\Omega_L &= \left\{ x \in \Omega \,\big|\, x_1 < 0 \right\}, \\
\Omega_R &= \left\{ x \in \Omega \,\big|\, x_1 > 0 \right\}.
\end{split}
\nonumber
\end{equation}
Below we assume
\begin{equation}
\begin{split}
f^L &~\text{is}~ C^k ~\text{in}~ \Omega_L \cup \Sigma, \\ 
f^R &~\text{is}~ C^k ~\text{in}~ \Omega_R \cup \Sigma, 
\end{split}
\label{eq:Ck}
\end{equation}
where different results use different values of $k \ge 0$.

Suppose $\bO \in \Sigma$ and $f^L(\bO) = \bO$ (that is, $\bO$ is a {\em boundary equilibrium} of \eqref{eq:f}).
Assuming $f^L$ and $f^R$ are $C^1$ at $x = \bO$, we can write
\begin{equation}
\begin{split}
f^L(x) &= A x + \co \left( \| x \| \right), \\
f^R(x) &= c + \cO \left( \| x \| \right),
\end{split}
\label{eq:fLfR}
\end{equation}
where $A = D f^L(\bO)$ is an $n \times n$ matrix and $c = f^R(\bO) \in \mathbb{R}^n$.
By removing the higher-order terms from \eqref{eq:fLfR} we obtain the approximation
\begin{equation}
\dot{x} = F(x) = \begin{cases}
A x, & x_1 < 0, \\
c, & x_1 > 0.
\end{cases}
\label{eq:F}
\end{equation}

\subsection{Main results}
\label{sub:results}

\begin{theorem}
Consider \eqref{eq:f} satisfying \eqref{eq:Ck} with $k = 1$ and \eqref{eq:fLfR} with $c_1 < 0$.
Then $\bO$ is an exponentially stable equilibrium of \eqref{eq:f}
if and only if it is an exponentially stable equilibrium of \eqref{eq:F}.
\label{th:expStabEquiv}
\end{theorem}

If $c_1 > 0$ then $x = \bO$ is unstable (i.e.~not Lyapunov stable)
for both \eqref{eq:f} and \eqref{eq:F}.
This is because there exists a Filippov solution in $\Omega_R$ that emanates from $\bO$ with direction $c$.
If $c_1 = 0$ then $x = \bO$ is not asymptotically stable for \eqref{eq:F}, but can be exponentially stable for \eqref{eq:f}.
For example this occurs for the system
\begin{align}
f^L(x) &= \begin{bmatrix} -\nu x_1 + x_2 \\ -x_1 - \nu x_2 \end{bmatrix}, &
f^R(x) &= \begin{bmatrix} x_2 \\ -1 \end{bmatrix},
\label{eq:c10}
\end{align}
with any $\nu > 0$, see Fig.~\ref{fig:counterExample}.

\begin{figure}[b!]
\begin{center}
\setlength{\unitlength}{1cm}
\begin{picture}(11,5)
\put(0,0){\includegraphics[height=4.5cm]{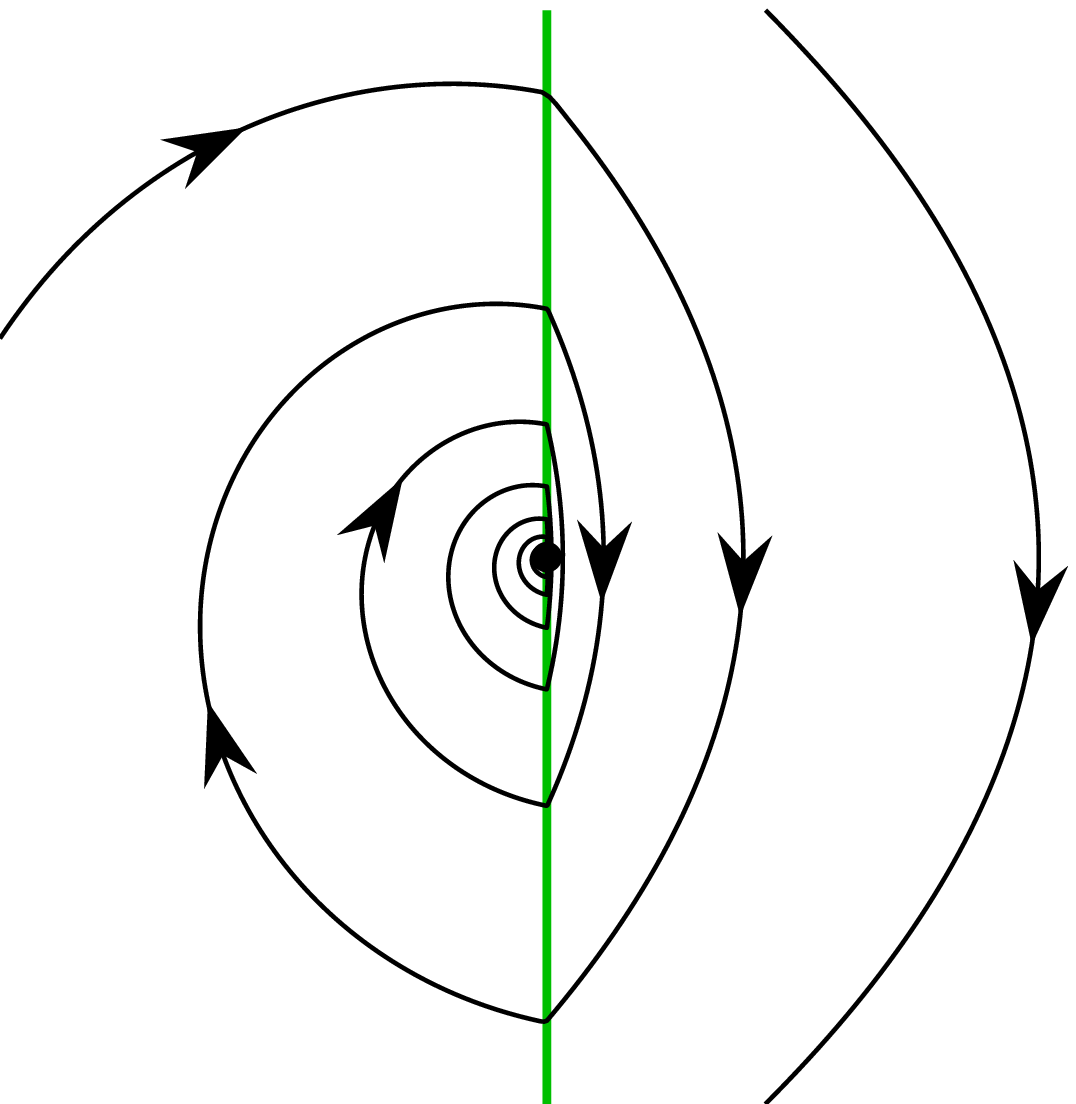}}
\put(6.5,0){\includegraphics[height=4.5cm]{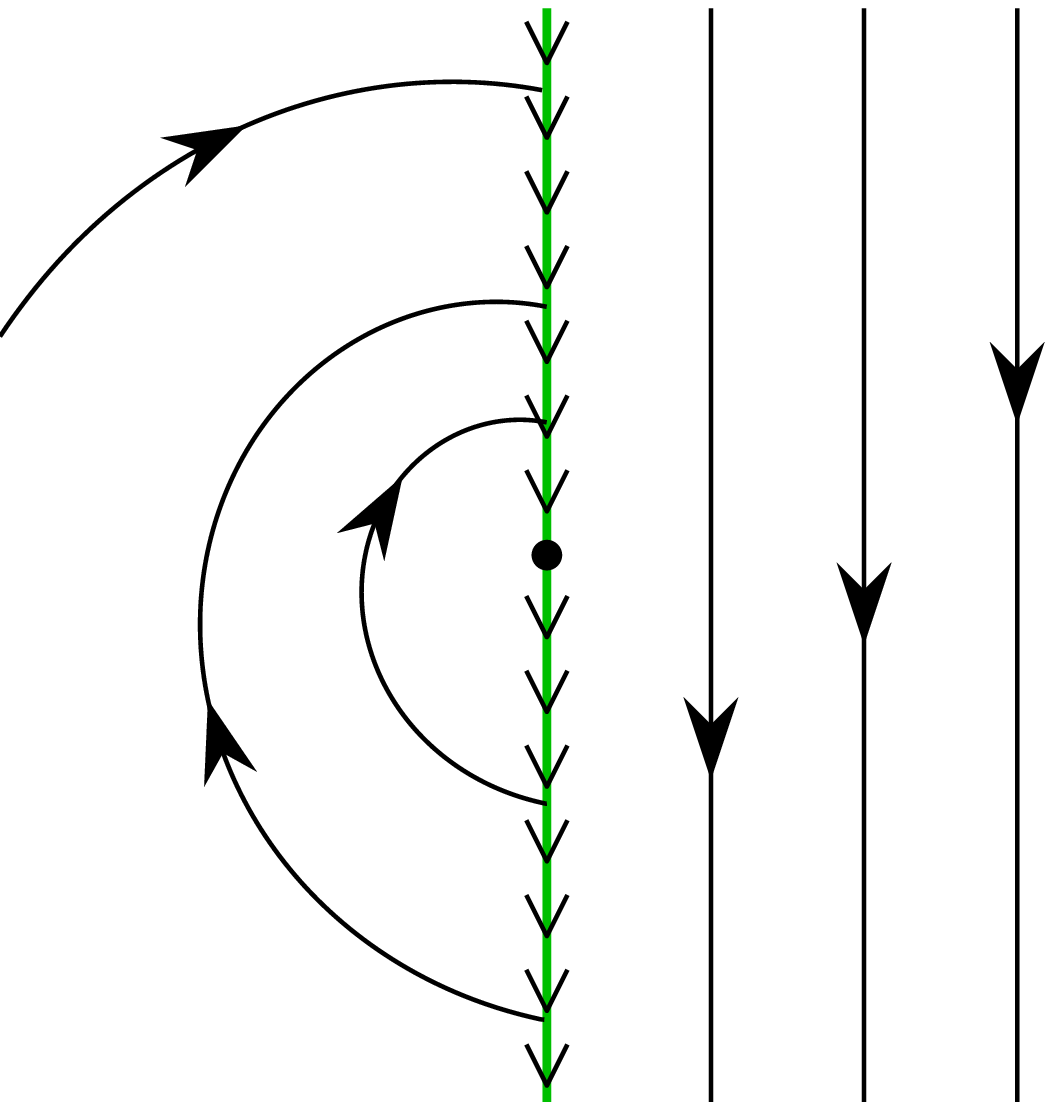}}
\put(1.8,4.7){\small $\dot{x} = f(x)$}
\put(8.3,4.7){\small $\dot{x} = F(x)$}
\end{picture}
\caption{
The left plot is a phase portrait of \eqref{eq:f} with \eqref{eq:c10} and $\nu = 0.2$.
Here the origin is exponentially stable.
The right plot is a phase portrait of the corresponding reduced system \eqref{eq:F}
(given by replacing $x_2$ in $f^R(x)$ with $0$).
Here the origin is unstable.
This example does not contradict Theorem \ref{th:expStabEquiv} because $c_1 = 0$.
\label{fig:counterExample}
} 
\end{center}
\end{figure}

Exponential stability is in general a stronger property than asymptotic stability,
but for \eqref{eq:F} these properties are in fact equivalent:

\begin{theorem}
The equilibrium $x = \bO$ of \eqref{eq:F} is exponentially stable if and only if it is asymptotically stable.
\label{th:asyStabImpliesExpStab}
\end{theorem}

Theorems \ref{th:expStabEquiv} and \ref{th:asyStabImpliesExpStab}
reduce the problem of the exponential stability of a boundary equilibrium of \eqref{eq:f}
to the asymptotic stability of $x = \bO$ for the simpler system \eqref{eq:F}.

Finally we show that exponential stability is robust to the entries in $A$ and $c$.
Consider a family of systems of the form \eqref{eq:F}:
\begin{equation}
\dot{x} = F_\mu(x) = \begin{cases}
A(\mu) x, & x_1 < 0, \\
c(\mu), & x_1 > 0,
\end{cases}
\label{eq:Ffamily}
\end{equation}
where $\mu$ is a parameter that takes values in a set $J$.

\begin{theorem}
Let $\tilde{\mu} \in J$ and suppose $\bO$ is exponentially stable for $\dot{x} = F_{\tilde{\mu}}(x)$.
There exists $\delta > 0$ such that if
$\| A(\mu) - A(\tilde{\mu}) \| < \delta$ (using the induced matrix norm)
and $\| c(\mu) - c(\tilde{\mu}) \| < \delta$,
then $\bO$ is also exponentially stable for $\dot{x} = F_\mu(x)$.
\label{th:robust}
\end{theorem}

\section{Existence, uniqueness, and continuity with respect to parameters}
\label{sec:euc}
\setcounter{equation}{0}

In this section we state three fundamental theorems regarding Filippov solutions for systems of the form \eqref{eq:f}.
These theorems are used below in \S\ref{sec:mainProof} to prove Theorems \ref{th:expStabEquiv}--\ref{th:robust}.
Proofs of the theorems of this section can be found in \cite{Fi88} (note that in \cite{Fi88} they are stated with more generality).

The first theorem ensures Filippov solutions exist, see \cite[pg.~85]{Fi88}.

\begin{theorem}
Consider \eqref{eq:f} satisfying \eqref{eq:Ck} with $k = 0$.
For any $x \in \Omega$, there exists $T > 0$ such that \eqref{eq:f}
has a Filippov solution $\phi(t)$ on $[-T,T]$ with $\phi(0) = x$.
\label{th:FilippovExistence}
\end{theorem}

The second theorem gives conditions under which Filippov solutions are unique forwards in time, see \cite[pg.~110]{Fi88}.
It shows that forward uniqueness can only be violated if a solution reaches a point on the discontinuity surface $\Sigma$
where neither $f^L$ nor $f^R$ is directed towards $\Sigma$.

\begin{theorem}
Consider \eqref{eq:f} satisfying \eqref{eq:Ck} with $k = 1$.
Suppose that for all $x \in \Sigma$ we have either $f^L_1(x) > 0$ or $f^R_1(x) < 0$.
Let $T > 0$ and let $\phi(t)$ and $\psi(t)$ be Filippov solutions to \eqref{eq:f}
that belong to $\Omega$ for all $t \in [0,T]$.
If $\phi(0) = \psi(0)$ then $\phi(t) = \psi(t)$ for all $t \in [0,T]$.
\label{th:FilippovUniqueness}
\end{theorem}

The final theorem shows that Filippov solutions vary continuously with respect to parameters, see \cite[pgs.~89--90]{Fi88}.
Let $\mu$ be a parameter that takes values in a set $J$, and consider a family of systems
\begin{equation}
\dot{x} = f_\mu(x) = \begin{cases}
f^L_\mu(x), & x_1 < 0, \\
f^R_\mu(x), & x_1 > 0.
\end{cases}
\label{eq:fmu}
\end{equation}

\begin{theorem}
Suppose $f_\mu$ satisfies \eqref{eq:Ck} with $k = 0$ for all $\mu \in J$.
Let $U \subset \Omega$ be compact, let $[a,b]$ be an interval containing $0$, and let $\tilde{\mu} \in J$.
Suppose all Filippov solutions to $\dot{x} = f_{\tilde{\mu}}(x)$ with initial conditions in $U$ belong to $\Omega$ for all $t \in [a,b]$.
Then for all $\ee > 0$ there exists $\delta > 0$ such that if $\| f_\mu(x) - f_{\tilde{\mu}}(x) \| < \delta$ for all $x \in \Omega \setminus \Sigma$, then
for any Filippov solution $\phi(t)$ to $\dot{x} = f_{\tilde{\mu}}(x)$ with $\phi(0) \in U$,
there exists a Filippov solution $\xi(t)$ to $\dot{x} = f_\mu(x)$ with $\xi(0) = \phi(0)$
and $\left\| \xi(t) - \phi(t) \right\| < \ee$ for all $t \in [a,b]$.
\label{th:FilippovContinuity}
\end{theorem}

\section{Sliding motion}
\label{sec:sliding}
\setcounter{equation}{0}

In this section we introduce the sliding vector field for characterising sliding motion on the discontinuity surface $\Sigma$.
The concepts presented here are quite elementary (see for instance \cite{DiBu08,Fi88,Je18b} for further discussion)
but we also prove a lemma specific to systems of the form \eqref{eq:f}
that formally connects Filippov solutions to the sliding vector field. 

Connected subsets of $\Sigma$ for which $f^L_1(x) f^R_1(x) > 0$ are termed {\em crossing regions}.
In view of Theorem \ref{th:FilippovUniqueness}, forward evolution is unique in the neighbourhood of a crossing region.
When a Filippov solution reaches a crossing region,
it simply crosses $\Sigma$ in a continuous fashion, see Fig.~\ref{fig:expStabSchem_b}.
We denote the union of all crossing regions by
\begin{equation}
\Sigma_{\rm cross} = \left\{ x \in \Sigma \,\big|\, f^L_1(x) f^R_1(x) > 0 \right\}.
\nonumber
\end{equation}

\begin{figure}[b!]
\begin{center}
\includegraphics[height=6cm]{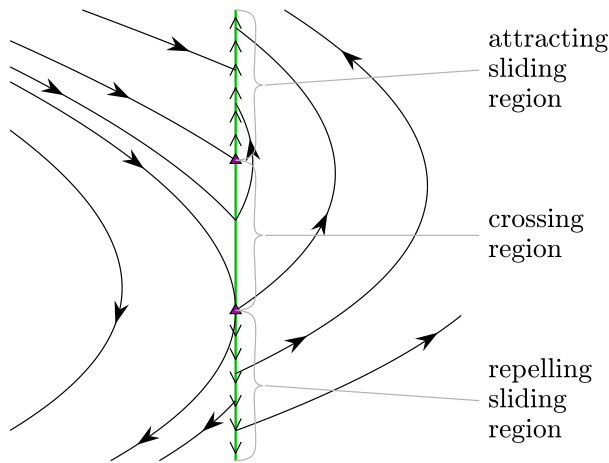}
\caption{
A phase portrait of a two-dimensional Filippov system of the form \eqref{eq:f}.
This system has two tangency points (purple triangles)
that divide the discontinuity surface $\Sigma$
into a crossing region and attracting and repelling sliding regions.
\label{fig:expStabSchem_b}
} 
\end{center}
\end{figure}

Connected subsets of $\Sigma$ for which $f^L_1(x) > 0$ and $f^R_1(x) < 0$
are termed {\em attracting sliding regions}.
Again forward evolution is unique but when a Filippov solution reaches an attracting sliding region
it subsequently evolves on $\Sigma$.
Such sliding motion also occurs on {\em repelling sliding regions}, where $f^L_1(x) < 0$ and $f^R_1(x) > 0$,
but here forward evolution is non-unique.

Sliding motion is specified through the set-valued function \eqref{eq:cFgeneral}.
For our system this function is given by
\begin{equation}
\cF(x) = \begin{cases}
\left\{ f^L(x) \right\}, & x_1 < 0, \\
\left\{ (1-\lambda) f^L(x) + \lambda f^R(x) \,\middle|\, 0 \le \lambda \le 1 \right\}, & x_1 = 0, \\
\left\{ f^R(x) \right\}, & x_1 > 0.
\end{cases}
\label{eq:cF}
\end{equation}
Since a sliding solution is constrained to $\Sigma$, its velocity at any $x \in \Sigma$
belongs to the set $\left\{ (1-\lambda) f^L(x) + \lambda f^R(x) \,\middle|\, 0 \le \lambda \le 1 \right\}$
and is tangent to $\Sigma$.
Thus the velocity, given by
\begin{equation}
f^S(x) = \frac{f^L_1(x) f^R(x) - f^R_1(x) f^L(x)}{f^L_1(x) - f^R_1(x)},
\label{eq:fS}
\end{equation}
exists and is unique at any point $x \in \Sigma_{\rm slide}$ where
\begin{equation}
\Sigma_{\rm slide} = \left\{ x \in \Sigma \,\big|\, f^L_1(x) f^R_1(x) \le 0,
~\text{with $f^L_1(x)$ and $f^R_1(x)$ not both zero} \right\}.
\nonumber
\end{equation}
The condition $f^L_1(x) f^R_1(x) \le 0$ ensures $0 \le \lambda \le 1$.
We exclude $f^L_1(x) = f^R_1(x) = 0$ because in this case $f^S(x)$ is not defined.
Equation \eqref{eq:fS} defines a vector field on sliding regions and is known as the {\em sliding vector field}.

Lastly we prove the following result that justifies \eqref{eq:fS}.
Informally, Lemma \ref{le:dphidt} says that Filippov solutions to \eqref{eq:f}
evolve according to $f^L$ while in $\Omega_L$,
evolve according to $f^R$ while in $\Omega_R$,
and slide on $\Sigma$ according to $f^S$.
Note that Filippov solutions are non-differentiable at $x \in \Sigma_{\rm cross}$
except in the special case $f^L(x) = f^R(x)$.
Also \eqref{eq:dphidt} does not hold at points for which $f^L_1(x) = f^R_1(x) = 0$
(such as {\em two-folds} \cite{Je18b})
because here $f^S(x)$ is not defined.

\begin{lemma}
Let $\phi : [a,b] \to \Omega$ be a Filippov solution to \eqref{eq:f} satisfying \eqref{eq:Ck} with $k = 0$.
Let $t \in (a,b)$
be such that $\phi$ is differentiable at $t$.
If $\phi_1(t) = 0$ suppose $f^L_1(\phi(t)) \ne f^R_1(\phi(t))$.
Then
\begin{equation}
\dot{\phi}(t) = \begin{cases}
f^L(\phi(t)), & \phi_1(t) < 0, \\
f^S(\phi(t)), & \phi_1(t) = 0, \\
f^R(\phi(t)), & \phi_1(t) > 0,
\end{cases}
\label{eq:dphidt}
\end{equation}
with $\phi(t) \in \Sigma_{\rm slide}$ in the case $\phi_1(t) = 0$.
\label{le:dphidt}
\end{lemma}

\begin{proof}
Let $x = \phi(t)$ and $v = \dot{\phi}(t)$.
Suppose for a contradiction that $v \notin \cF(x)$.
Since $f^L$ and $f^R$ are continuous
there exists a neighbourhood $N$ of $x$ such that $v \notin K$ where
\begin{equation}
K = {\rm co} \left[ \bigcup_{y \in N} \cF(y) \right].
\nonumber
\end{equation}
Since $\phi$ is continuous there exists $\delta > 0$ such that $\phi(s) \in N$
for all $s \in (t-\delta,t+\delta)$.
But $\dot{\phi}(t) \notin K$,
thus there exists $s \in (t,t+\delta)$ such that
$\phi(s) = x + (s-t) p$ where $p \notin K$.

Next we use the knowledge that $\phi$ is a Filippov solution
to write $\phi(s) = x + (s-t) q$ where $q \in K$, giving a contradiction.
Since $\phi$ is absolutely continuous, we can write
\begin{equation}
\phi(s) = x + \int_t^s u(\tilde{t}) \,d\tilde{t},
\label{eq:dphidtProof20}
\end{equation}
where $u$ is Lebesgue integrable \cite[pg.~125]{RoFi10}.
Since $\phi$ is a Filippov solution we have
$u(\tilde{t}) \in \cF(\phi(\tilde{t})) \subset K$ for almost all $\tilde{t} \in (t-\delta,t+\delta)$.
Since $K$ is convex, by \eqref{eq:dphidtProof20} there exists $q \in K$ such that $\phi(s) = x + (s-t) q$
(this follows from a version of the mean value theorem, see for instance \cite[pg.~63]{Fi88}).
This is a contradiction, hence $v \in \cF(x)$.

If $x_1 < 0$ [resp.~$x_1 > 0$] then \eqref{eq:cF} immediately gives $v = f^L(x)$ [resp.~$v = f^R(x)$].
It remains for us to consider the case $x_1 = 0$.
Thus we can assume $f^L_1(x) \ne f^R_1(x)$.
Suppose for a contradiction that $v_1 > 0$.
Then $\dot{\phi}(t) = v$ implies there exists $\delta > 0$ such that $\phi_1(s) > 0$ for all $s \in (t,t+\delta)$.
But, using the form \eqref{eq:dphidtProof20},
then $u(\tilde{t}) = f^R(\phi(\tilde{t}))$ for almost all $\tilde{t} \in (t,t+\delta)$.
By taking $s \to t$ from above we obtain $v = f^R(x)$ from the continuity of $f^R$.
By similarly taking $s \to t$ from below we obtain $v = f^L(x)$.
But $f^L_1(x) \ne f^R_1(x)$ so we have a contradiction.
For similar reasons we cannot have $v_1 < 0$.
Therefore $v_1 = 0$.
Finally, since $v \in \cF(x)$
and $f^L_1(x) \ne f^R_1(x)$, we have $v = f^S(x)$ and $x \in \Sigma_{\rm slide}$.
\end{proof}

\section{The approximate system as a time-scaled linearly homogeneous system}
\label{sec:lh}
\setcounter{equation}{0}

In this section we provide intuition to Theorem \ref{th:asyStabImpliesExpStab}
by considering simple scalings of time and space.
These scaling are also central to the proofs of Theorems \ref{th:expStabEquiv} and \ref{th:robust}.

A vector field $g : \mathbb{R}^n \to \mathbb{R}^n$ is {\em linearly homogeneous}
(or {\em homogeneous with degree $1$}) if
$g(\gamma x) = \gamma g(x)$ for all $x \in \mathbb{R}^n$ and all $\gamma > 0$.
If $g$ is linearly homogeneous and $\xi(t)$ is a solution to $\dot{x} = g(x)$,
then, for any $\gamma > 0$, $\gamma \xi(t)$ is also a solution to $\dot{x} = g(x)$.
That is, solutions to $\dot{x} = g(x)$ behave in the same way at all spatial scales.
Solutions converge to (or diverge from) the origin at a rate that is independent of the distance from the origin.
Due to this property, if $\bO$ is asymptotically stable then it is also exponentially stable.

If $c_1 < 0$ then Filippov solutions to \eqref{eq:F} in $\Omega_L \cup \Sigma$
have the same property but with a time scaling.
To see this, first observe that if $c_1 < 0$ then Filippov solutions cannot escape $\Omega_L \cup \Sigma$ (using $\Omega = \mathbb{R}^n$).
Thus, by Lemma \ref{le:dphidt}, they evolve via $F^L(x) = A x$ and the sliding vector field of \eqref{eq:F}, call it $F^S(x)$.
Notice $F^L(x)$ is a linear system, while
\begin{equation}
F^S(x) = \frac{\left( I - \frac{c e_1^{\sf T}}{c_1} \right) A x}{1 - \frac{e_1^{\sf T} A x}{c_1}},
\label{eq:FS}
\end{equation}
obtained by evaluating \eqref{eq:fS}, is a scalar multiple of a linear system.
Intuitively the scalar multiple can be removed by an appropriate rescaling of time $t \mapsto s$,
resulting in the description
\begin{equation}
\frac{d x}{d s} = \begin{cases}
A x, & \text{while}~ x_1 < 0, \\
C x, & \text{while}~ x_1 = 0,
\end{cases}
\label{eq:dxds}
\end{equation}
where $C = \left( I - \frac{c e_1^{\sf T}}{c_1} \right) A$.
The system \eqref{eq:dxds} is linearly homogeneous,
from which we infer that solutions behave in the same way at all spatial scales, but with a mild adjustment to the speed of evolution.
This loosely explains why asymptotic stability implies exponential stability for \eqref{eq:F}. 

But on its own \eqref{eq:dxds} is ill-defined.
Theorem \ref{th:asyStabImpliesExpStab} can be proved
by first performing the spatial scaling $z = \frac{x}{\gamma}$, where $\gamma > 0$.
This transforms \eqref{eq:F} to
\begin{equation}
\dot{z} = \frac{1}{\gamma} F(\gamma z) = \begin{cases}
A z, & z_1 < 0, \\
\frac{1}{\gamma} c, & z_1 > 0.
\end{cases}
\label{eq:Fscaled}
\end{equation}
On the other hand we can arrive at \eqref{eq:Fscaled} by scaling time by $\frac{1}{\gamma}$ in the right half-system of \eqref{eq:F}.
This spatially-discontinuous time scaling is addressed in the next section.

\section{A discontinuous time scaling}
\label{sec:timeScaling}
\setcounter{equation}{0}

Consider a system of the form \eqref{eq:f}.
Given $\gamma \in (0,1]$ we construct the scaled system
\begin{equation}
\dot{x} = \begin{cases}
f^L(x), & x_1 < 0, \\
\frac{1}{\gamma} f^R(x), & x_1 > 0.
\end{cases}
\label{eq:fCheck}
\end{equation}
In this section we prove the following lemma
that connects Filippov solutions of \eqref{eq:f} to those of \eqref{eq:fCheck} (and vice-versa).
A similar result can be found in \cite[pg.~102]{Fi88}
for a system scaled by a continuous function of $x$.
In \eqref{eq:fCheck} time is scaled by a piecewise-constant function that is discontinuous on $\Sigma$.
In view of this discontinuity we require arguments additional to those given in \cite{Fi88} in order to prove Lemma \ref{le:discTimeScaling}.

\begin{lemma}
Suppose \eqref{eq:f} satisfies \eqref{eq:Ck} with $k = 1$.
Suppose $f^R_1(x) < 0$ for all $x \in \Sigma$.
\begin{enumerate}
\item
If $\phi : [0,T] \to \Omega$ is a Filippov solution to \eqref{eq:f}
then there exists a strictly increasing, Lipschitz function $p : [0,T] \to \mathbb{R}$,
with $\gamma t \le p(t) \le t$ for all $t \in [0,T]$, such that $\phi \left( p^{-1}(s) \right)$
is a Filippov solution to \eqref{eq:fCheck} for $s \in [0,p(T)]$.
\item
Conversely if $\psi : [0,S] \to \Omega$ is a Filippov solution to \eqref{eq:fCheck}
then such a $p$ exists so that $\psi(p(t))$ is a Filippov solution to \eqref{eq:f} for $t \in \left[ 0,p^{-1}(S) \right]$.
\item
Further, if the initial point of the given Filippov solution lies in $\Omega_L \cup \Sigma$
and $\left| \frac{f^L_1(x)}{f^R_1(x)} \right| \le M$ for all $x \in \Sigma_{\rm slide}$,
then $p(t) \ge \frac{t}{M+1}$ for all $t \in [0,T]$.
\end{enumerate}
\label{le:discTimeScaling}
\end{lemma}

Notice that part (iii) of Lemma \ref{le:discTimeScaling} provides a lower bound on $p(t)$ that,
unlike the bound in part (i), does not tend to $0$ as $\gamma \to 0$.

\begin{proof}
Here we prove parts (i) and (iii).
For brevity a proof of part (ii) is omitted as it can be proved in the same way as part (i).

\myStep{1}{Compare the sliding vector fields of \eqref{eq:f} and \eqref{eq:fCheck}.}
Let $h^L = f^L$ and $h^R = \frac{1}{\gamma} f^R$ denote the components of \eqref{eq:fCheck}.
Let $h^S$ denote the sliding vector field for \eqref{eq:fCheck}.
Then for each $Z \in \{ L, R, S \}$ we have $f^Z(x) = a(x) h^Z(x)$ where
\begin{equation}
a(x) = \begin{cases}
1, & x_1 < 0, \\
a_0(x), & x_1 = 0, \\
\gamma, & x_1 > 0,
\end{cases}
\label{eq:alpha}
\end{equation}
and
\begin{equation}
a_0(x) = \frac{\gamma f^L_1(x) - f^R_1(x)}{f^L_1(x) - f^R_1(x)}.
\label{eq:alpha0}
\end{equation}
Notice that $\gamma < a_0(x) \le 1$ for any $x \in \Sigma_{\rm slide}$.
The condition $f^R_1 < 0$ implies $\Sigma = \Sigma_{\rm slide} \cup \Sigma_{\rm cross}$
and that $\phi(t)$ intersects $\Sigma_{\rm cross}$ at most once.

\myStep{2}{Construct $p$.}
Let $\Lambda$ be the set of all $t \in (0,T)$ for which $\phi(t)$ 
is differentiable and $\phi(t) \notin \Sigma_{\rm cross}$.
Then \eqref{eq:dphidt} holds for all $t \in \Lambda$ and if $\phi_1(t) = 0$ then $\phi(t) \in \Sigma_{\rm slide}$.
Observe $\Lambda$ has full measure on $[0,T]$ because $\phi$ is absolutely continuous.

We now show $a(\phi(t))$ is continuous at all $t \in \Lambda$.
If $\phi_1(t) \ne 0$ then $a(\phi(t))$ is continuous at $t$ because $a$ is continuous at $\phi(t)$.
If $\phi_1(t) = 0$ then $\dot{\phi}_1(t) = 0$ by \eqref{eq:dphidt} and there are the following two cases.
If $f^L_1(\phi(t)) > 0$ then there exists $\delta > 0$ such that $\phi_1(s) = 0$ for all $s \in (t-\delta,t+\delta)$,
and so $a(\phi(t))$ is continuous at $t$ because $a_0$ is continuous.
Otherwise $f^L_1(\phi(t)) = 0$ in which case $a(\phi(t))$ is continuous at $t$ because $a_0(\phi(t)) = 1$.

Consequently $a(\phi(t))$ is measurable by Lusin's theorem \cite[pg.~518]{So14}
and so also Lebesgue integrable \cite[pg.~74]{RoFi10}.
Therefore we can define $p : [0,T] \to \mathbb{R}$ by
\begin{equation}
p(t) = \int_0^t a(\phi(\tilde{t})) \,d\tilde{t},
\label{eq:u}
\end{equation}
and $\frac{d p}{d t} = a(\phi(t))$ at all $t \in \Lambda$.
Since $\gamma \le a_0(\phi(t)) \le 1$ for all $t \in [0,T]$,
we have $\gamma (t-s) \le p(t) - p(s) \le t-s$ for all $t \in [0,T]$ with $t \ge s$,
thus $p$ satisfies the properties stated in part (i).

\myStep{3}{Form $q = p^{-1}$ and prove part (iii).}
Thus $p$ has a continuous and strictly increasing inverse $q : [0,p(T)] \to [0,T]$
and $q$ is Lipschitz with constant $\frac{1}{\gamma}$:
\begin{equation}
q(t) - q(s) \le \frac{1}{\gamma}(t - s), \qquad \text{for all~} 0 \le s \le t \le p(T).
\label{eq:qLipschitz}
\end{equation}
Also with $M$ as given in part (iii),
we have $a_0(\phi(t)) \ge \frac{1}{1+M}$ for all $t \in \Lambda$ for which $\phi_1(t) = 0$.
Thus $1+M$ is also a Lipschitz constant for $q$,
thus $p(t) \ge \frac{t}{M+1}$ for all $t \in [0,T]$.

\myStep{4}{Show $\phi(q(t))$ is a Filippov solution to \eqref{eq:fCheck}.}
Let $\psi(t) = \phi(q(t))$.
We first show $\psi(t)$ is absolutely continuous.
Choose any $\ee > 0$.
Since $\phi$ is absolutely continuous there exists $\delta > 0$ such that
for any pairwise disjoint sub-intervals $[\hat{s}_k,\hat{t}_k] \subset [0,T]$, if $\sum_k (\hat{t}_k-\hat{s}_k) < \frac{\delta}{\gamma}$
then $\sum_k \| \phi(\hat{t}_k) - \phi(\hat{s}_k) \| < \ee$.
Now choose any pairwise disjoint sub-intervals $[s_k,t_k] \subset [0,p(T)]$ for which $\sum_k (t_k - s_k) < \delta$.
By the Lipschitz property \eqref{eq:qLipschitz} we have $q(t_k) - q(s_k) < \frac{1}{\gamma}(t_k - s_k)$ for each $k$,
and the intervals $[q(s_k),q(t_k)]$ are pairwise disjoint because $q$ is strictly increasing,
thus $\sum_k \left( q(t_k) - q(s_k) \right) < \frac{\delta}{\gamma}$
and hence $\sum_k \| \psi(t_k) - \psi(s_k) \| < \ee$.
Thus $\psi(t)$ is absolutely continuous.

Finally, $p(\Lambda)$ has full measure on $[0,p(T)]$ because $p$ is absolutely continuous
so satisfies the Lusin $N$ property \cite[pg.~388]{Bo07}.
For any $t \in p(\Lambda)$, we have $q(t) \in \Lambda$
so $\frac{d q}{d t} = \frac{1}{a(\phi(q(t)))} = \frac{1}{a(\psi(t))}$.
Thus $\dot{\psi}(t) = \dot{\phi}(q(t)) \frac{d q}{d t} = \frac{f^Z(\psi(t))}{a(\psi(t))}$
for the index $Z \in \{ L, R, S \}$ specified by Lemma \ref{le:dphidt}.
This is equal to $h^Z(\psi(t))$, therefore $\psi(t)$ is a Filippov solution to \eqref{eq:fCheck} for $s \in [0,p(T)]$.
\end{proof}

\section{Proofs of Theorems \ref{th:expStabEquiv}--\ref{th:robust}}
\label{sec:mainProof}
\setcounter{equation}{0}

If \eqref{eq:f} satisfies the conditions of Theorem \ref{th:FilippovUniqueness},
then \eqref{eq:f} generates a unique {\em semi-flow} $\varphi_t(x)$.
That is, for each $x \in \Omega$, $\varphi_t(x) \in \Omega$ is the Filippov solution to \eqref{eq:f} with initial point $x$ (i.e.~$\varphi_0(x) = x$).
This solution is defined for all $t \in [0,T]$, for some $T > 0$ that in general depends on $x$.
The semi-flow satisfies the group property
\begin{equation}
\varphi_{s+t}(x) = \varphi_s(\varphi_t(x)),
\label{eq:semiFlowGroupProperty}
\end{equation}
for all $s, t \ge 0$ for which $s + t \le T$.

\begin{lemma}
Suppose $x = \bO$ is an asymptotically stable equilibrium of a semi-flow $\varphi_t(x)$.
Then $\varphi_t(x) \to \bO$ uniformly on some $B_\delta(\bO)$
(that is, for all $\ee > 0$ there exists $T > 0$ such that $\varphi_t(x) \in B_\ee(\bO)$ for all $x \in B_\delta(\bO)$ and all $t \ge T$).
\label{le:uniformConvergence}
\end{lemma}

This result is proved in \cite{DiNo08} by first establishing equicontinuity of $\varphi_t(x)$.
In Appendix \ref{app:uniformConvergence} we provide a proof of Lemma \ref{le:uniformConvergence}
that avoids equicontinuity by introducing an additional small quantity $\delta_1 > 0$.
A generalisation of Lemma \ref{le:uniformConvergence} that allows
non-unique forward evolution is proved by contradiction in \cite[pg.~160]{Fi88}.

Next we prove Theorem \ref{th:asyStabImpliesExpStab}.
We then prove Theorems \ref{th:expStabEquiv} and \ref{th:robust} together.

\begin{proof}[Proof of Theorem \ref{th:asyStabImpliesExpStab}]~\\
We suppose $x = \bO$ is asymptotically stable and show it is also exponentially stable
(the converse is trivial by Definition \ref{df:stable}).
Notice asymptotic stability implies $c_1 < 0$
(this enables us to apply Lemma \ref{le:discTimeScaling}).

Let $\psi_t(x)$ denote the semi-flow of \eqref{eq:F}
(defined for all $x \in \mathbb{R}^n$ and all $t \ge 0$ by Theorem \ref{th:FilippovUniqueness}).
Since $x = \bO$ is asymptotically stable there exists $\delta > 0$ such that
\begin{equation}
\psi_t(x) \in B_1(\bO), \qquad \text{for all}~ x \in \overline{B}_\delta(\bO) ~\text{and all}~ t \ge 0,
\label{eq:asyStabProofCond1}
\end{equation}
and $\psi_t(x) \to \bO$ as $t \to \infty$ for all $x \in \overline{B}_\delta(\bO)$.
This convergence is uniform (Lemma \ref{le:uniformConvergence}) thus there exists $T > 0$ such that
\begin{equation}
\psi_t(x) \in B_{\frac{\delta}{2}}(\bO), \qquad \text{for all}~ x \in \overline{B}_\delta(\bO) ~\text{and all}~ t \ge T.
\label{eq:asyStabProofCond2}
\end{equation}

Choose any $x \in \overline{B}_\delta(\bO)$, with $x \ne \bO$, and let $\gamma = \frac{\| x \|}{\delta} \in (0,1]$.
By part (i) of Lemma \ref{le:discTimeScaling} applied to the point $\frac{x}{\gamma} \in \overline{B}_\delta(\bO)$ and the end-time $\frac{T}{\gamma}$,
there exists a strictly increasing, Lipschitz function $p : \left[ 0, \frac{T}{\gamma} \right] \to \mathbb{R}$,
with $\gamma t \le p(t) \le t$ for all $t \in \left[ 0, \frac{T}{\gamma} \right]$,
such that $\psi_{p^{-1}(s)} \big( \frac{x}{\gamma} \big)$ is a solution to \eqref{eq:Fscaled}.
Notice $p^{-1}(s)$ is defined for at least all $s \in [0,T]$.
But \eqref{eq:Fscaled} is obtained via the scaling $z = \frac{x}{\gamma}$,
therefore $\frac{1}{\gamma} \psi_s(x)$ is also a Filippov solution to \eqref{eq:Fscaled}.
These two solutions have the same initial point, thus, by the uniqueness of the semi-flow, they coincide.
Therefore $\left\| \psi_s(x) \right\| = \gamma \left\| \psi_{p^{-1}(s)} \big( \frac{x}{\gamma} \big) \right\|$ for all $s \in [0,T]$.
Then by \eqref{eq:asyStabProofCond1} applied to the point $\frac{x}{\gamma} \in \overline{B}_\delta(\bO)$ we have
\begin{equation}
\left\| \psi_t(x) \right\| \le \gamma = \frac{1}{\delta} \| x \|, \qquad \text{for all}~ x \in \overline{B}_\delta(\bO) ~\text{and all}~ t \in [0,T].
\label{eq:asyStabProofCond3}
\end{equation}
Similarly by \eqref{eq:asyStabProofCond2} we have
\begin{equation}
\left\| \psi_T(x) \right\| \le \frac{\gamma \delta}{2} = \frac{1}{2} \| x \|, \qquad \text{for all}~ x \in \overline{B}_\delta(\bO).
\label{eq:asyStabProofCond4}
\end{equation}
By applying \eqref{eq:asyStabProofCond3} and \eqref{eq:asyStabProofCond4} recursively we obtain
$\left\| \psi_t(x) \right\| \le \frac{1}{2^k \delta} \,\| x \|$
for all $x \in \overline{B}_\delta(\bO)$,
all integers $k \ge 0$,
and all $k T \le t \le (k+1) T$.
Thus $x = \bO$ is exponentially stable with
$\alpha = \frac{2}{\delta}$ and $\beta = \frac{\ln(2)}{T}$
in Definition \ref{df:stable}.
\end{proof}

\begin{proof}[Proof of Theorems \ref{th:expStabEquiv} and \ref{th:robust}]~\\
Let $\dot{x} = f(x)$ and $\dot{x} = g(x)$ be Filippov systems of the form \eqref{eq:f} satisfying \eqref{eq:Ck} with $k = 1$.
Write
\begin{align}
f(x) &= \begin{cases}
A x + E^L_f(x), & x_1 < 0, \\
c + E^R_f(x), & x_1 > 0,
\end{cases}
\label{eq:fProof} \\
g(x) &= \begin{cases}
\tilde{A} x + E^L_g(x), & x_1 < 0, \\
\tilde{c} + E^R_g(x), & x_1 > 0,
\end{cases}
\label{eq:gProof}
\end{align}
where $c_1, \tilde{c}_1 < 0$ and
\begin{equation}
\begin{split}
E^L_f(x), E^L_g(x) ~\text{are}~ \co(x), \\
E^R_f(x), E^R_g(x) ~\text{are}~ \cO(x).
\end{split}
\label{eq:errorTerms}
\end{equation}
Suppose $\bO$ is an exponentially stable equilibrium of $\dot{x} = f(x)$.
Below we show that $\bO$ is also exponentially stable for $\dot{x} = g(x)$,
assuming $\| A - \tilde{A} \|$ and $\| c - \tilde{c} \|$ are sufficiently small.
This will complete the proofs of Theorems \ref{th:expStabEquiv} and \ref{th:robust} because we may take
(i) $g$ to be the leading-order approximation to $f$,
(ii) $f$ to be the leading-order approximation to $g$, and
(iii) $f$ and $g$ to be instances of \eqref{eq:Ffamily}.
In (iii) to prove Theorem \ref{th:robust}
the assumption $\tilde{c}_1 < 0$ is justified because $c_1 < 0$ (by the asymptotic stability of $\bO$ for \eqref{eq:F})
and $\| c - \tilde{c} \|$ can be assumed to be smaller than $|c_1|$.

\myStep{1}{Initial consequences of exponential stability.}
Let $\varphi_t(x)$ and $\psi_t(x)$ denote the semi-flows of $\dot{x} = f(x)$ and $\dot{x} = g(x)$, respectively.
The exponential stability assumption implies there exist
$\alpha_0 \ge 1$, $\beta_0 > 0$, and $\delta_0 > 0$ such that
\begin{equation}
\| \varphi_t(x) \| \le \alpha_0 \re^{-\beta_0 t} \| x \|,
\qquad \text{for all}~ x \in B_{\delta_0}(\bO) ~\text{and all}~ t \ge 0.
\label{eq:expStabEquivProof0}
\end{equation}
Assume $\delta_0$ is small enough that $B_{2 \delta_0 \alpha_0}(\bO) \subset \Omega$.
Let $f^L$ and $f^R$ denote the left and right half-systems of \eqref{eq:fProof}
and let $g^L$ and $g^R$ denote the left and right half-systems of \eqref{eq:gProof}.
Then there exists $M > 0$ such that
\begin{equation}
\frac{2 \alpha_0 \left| f^L_1(x) \right|}{\| x \| \left| f^R_1(x) \right|},
\frac{2 \alpha_0 \left| g^L_1(x) \right|}{\| x \| \left| g^R_1(x) \right|} \le M,
\qquad \text{for all}~ x \in B_{2 \delta_0 \alpha_0}(\bO).
\label{eq:expStabEquivProof1}
\end{equation}
Let
\begin{equation}
T = \frac{(M+1) \ln( 4 \alpha_0)}{\beta_0}.
\label{eq:TProof}
\end{equation}

\myStep{2}{Apply spatial blow-up and scale right half-systems by $\gamma$.}
Given $0 < \gamma < \delta_0$ we consider the spatial scaling $z = \frac{x}{\gamma}$.
This transforms \eqref{eq:fProof} and \eqref{eq:gProof} to
\begin{align}
\dot{z} &= \frac{1}{\gamma} f(\gamma z) = \begin{cases}
A z + \frac{1}{\gamma} E^L_f(\gamma z), & z_1 < 0, \\
\frac{1}{\gamma} \left( c + E^R_f(\gamma z) \right), & z_1 > 0,
\end{cases}
\label{eq:fScaledProof} \\
\dot{z} &= \frac{1}{\gamma} g(\gamma z) = \begin{cases}
\tilde{A} z + \frac{1}{\gamma} E^L_g(\gamma z), & z_1 < 0, \\
\frac{1}{\gamma} \left( \tilde{c} + E^R_g(\gamma z) \right), & z_1 > 0,
\end{cases}
\label{eq:gScaledProof}
\end{align}
with semi-flows $\frac{1}{\gamma} \varphi_t(\gamma z)$ and $\frac{1}{\gamma} \psi_t(\gamma z)$, respectively.
By scaling their right half-systems by $\gamma$ we obtain two new systems
\begin{align}
\dot{z} &= \check{f}(z) = \begin{cases}
A z + \frac{1}{\gamma} E^L_f(\gamma z), & z_1 < 0, \\
c + E^R_f(\gamma z), & z_1 > 0,
\end{cases}
\label{eq:fCheckProof} \\
\dot{z} &= \check{g}(z) = \begin{cases}
\tilde{A} z + \frac{1}{\gamma} E^L_g(\gamma z), & z_1 < 0, \\
\tilde{c} + E^R_g(\gamma z), & z_1 > 0.
\end{cases}
\label{eq:gCheckProof}
\end{align}
We denote the semi-flows of \eqref{eq:fCheckProof} and \eqref{eq:gCheckProof}
by $\check{\varphi}_s(z)$ and $\check{\psi}_s(z)$, respectively.

\myStep{3}{Apply Lemma \ref{le:discTimeScaling} to \eqref{eq:fCheckProof}.}
Define the compact set
\begin{equation}
U = \left\{ z \in \overline{B}_1(\bO) \,\big|\, z_1 \le 0 \right\}.
\label{eq:UProof}
\end{equation}
By \eqref{eq:expStabEquivProof0}, for any $z \in U$
the Filippov solution to \eqref{eq:fScaledProof} with initial point $z$
satisfies $\big\| \frac{1}{\gamma} \varphi_t(\gamma z) \big\| \le \alpha_0 \re^{-\beta_0 t}$,
so is contained in $B_{2 \alpha_0}(\bO)$ for all $t \ge 0$.
Let $\check{f}^L$ and $\check{f}^R$ denote the left and right components of \eqref{eq:fCheckProof}.
For any $z \in B_{2 \alpha_0}(\bO)$ we have
\begin{equation}
\frac{\left| \check{f}^L_1(z) \right|}{\left| \check{f}^R_1(z) \right|}
= \frac{\left| f^L_1(\gamma z) \right|}{\gamma \left| f^R_1(\gamma z) \right|}
= \frac{\| z \| \left| f^L_1(x) \right|}{\| x \| \left| f^R_1(x) \right|} \le M,
\nonumber
\end{equation}
by \eqref{eq:expStabEquivProof1}.
By parts (ii) and (iii) of Lemma \ref{le:discTimeScaling} applied to \eqref{eq:fCheckProof}
on $B_{2 \alpha_0}(\bO)$,
for any $z \in U$ there exists a strictly increasing, Lipschitz function $p_1 : [0,\infty) \to \mathbb{R}$, with $\frac{s}{M+1} \le p_1(s) \le s$ for all $s \ge 0$,
such that $\frac{1}{\gamma} \varphi_{p_1(s)}(\gamma z)$ is a Filippov solution to \eqref{eq:fCheckProof}.
By uniqueness of the semi-flow,
$\check{\varphi}_s(z) = \frac{1}{\gamma} \varphi_{p_1(s)}(\gamma z)$.
Then by \eqref{eq:expStabEquivProof0} applied to $\gamma z \in B_{\delta_0}(\bO)$,
\begin{equation}
\| \check{\varphi}_s(z) \| \le \alpha_0 \re^{\frac{-\beta_0 s}{M+1}},
\qquad \text{for all}~ z \in U ~\text{and all}~ s \ge 0,
\label{eq:expStabEquivProof30}
\end{equation}
where we have also used $p_1(s) \ge \frac{s}{M+1}$.

\myStep{4}{Apply Theorem \ref{th:FilippovContinuity}.}
By \eqref{eq:expStabEquivProof30}, $\check{\varphi}_s(z) \in B_{2 \delta_0 \alpha_0}(\bO)$
for all $z$ in the compact set $U$ and all $s \in [0,(M+1)T]$.
Then by Theorem \ref{th:FilippovContinuity} there exists $\eta > 0$ such that
if $\left\| \check{f}(z) - \check{g}(z) \right\| < \eta$ for all $z \in B_{2 \delta_0 \alpha_0}(\bO) \setminus \Sigma$, then
\begin{equation}
\left\| \check{\varphi}_s(z) - \check{\psi}_s(z) \right\| < \frac{1}{4},
\qquad \text{for all}~z \in U ~\text{and all}~s \in [0,(M+1)T].
\label{eq:expStabEquivProof40}
\end{equation}
By \eqref{eq:errorTerms} there exists $\delta > 0$ such that
\begin{align*}
\frac{\left\| E^L_f(x) \right\|}{\| x \|}, \frac{\left\| E^L_g(x) \right\|}{\| x \|} &< \frac{\eta}{3}, 
\qquad \text{for all $x \in B_\delta(\bO) \setminus \{ \bO \}$ with $x_1 < 0$}, \\
\left\| E^R_f(x) \right\|, \left\| E^R_g(x) \right\|  &< \frac{\eta}{3}, 
\qquad \text{for all $x \in B_\delta(\bO)$ with $x_1 > 0$}.
\end{align*}
Now assume
$\big\| (A - \tilde{A}) z \big\| < \frac{\eta}{3}$ for all $z \in B_{2 \delta_0 \alpha_0}(\bO)$,
and $\| c - \tilde{c} \| < \frac{\eta}{3}$.
Then from \eqref{eq:fCheckProof}--\eqref{eq:gCheckProof},
\begin{equation}
\left\| \check{f}(z) - \check{g}(z) \right\| < \eta,
\qquad \text{for all}~ z \in U ~\text{and all}~ 0 < \gamma < \delta.
\nonumber
\end{equation}
Thus from \eqref{eq:expStabEquivProof30} and \eqref{eq:expStabEquivProof40},
\begin{align}
\left\| \check{\psi}_s(z) \right\| 
&\le \left\| \check{\varphi}_s(z) - \check{\psi}_s(z) \right\| + \left\| \check{\varphi}_s(z) \right\| \nonumber \\
&\le \frac{1}{4} + \alpha_0 \re^{\frac{-\beta_0 s}{M+1}},
\qquad \text{for all}~z \in U,~s \in [0,(M+1)T], ~\text{and}~ 0 < \gamma < \delta.
\label{eq:expStabEquivProof48}
\end{align}

\myStep{5}{Apply Lemma \ref{le:discTimeScaling} to \eqref{eq:gCheckProof}.}
By \eqref{eq:expStabEquivProof48}, $\check{\psi}_s(z) \in B_{2 \alpha_0}(\bO)$ for all $s \in [0,(M+1)T]$.
By parts (i) and (iii) of Lemma \ref{le:discTimeScaling} applied to \eqref{eq:gCheckProof} on $B_{2 \alpha_0}(\bO)$,
for any $z \in U$ there exists a strictly increasing, Lipschitz function $p_2 : [0,(M+1)T] \to \mathbb{R}$,
with $\frac{s}{M+1} \le p_2(s) \le s$ for all $s \in [0,(M+1)T]$,
such that $\check{\psi}_{p_2^{-1}(t)}(z)$ is a Filippov solution to \eqref{eq:fScaledProof}.
Observe $p_2((M+1)T) \ge T$, so $p_2^{-1}(t)$ is defined for at least all $t \in [0,T]$.
By uniqueness of the semi-flow,
$\frac{1}{\gamma} \psi_t(\gamma z) = \check{\psi}_{p_2^{-1}(t)}(z)$.
Then by \eqref{eq:expStabEquivProof48},
\begin{equation}
\left\| \psi_t(\gamma z) \right\| \le \gamma \left( \frac{1}{4} + \alpha_0 \re^{\frac{-\beta_0 t}{M+1}} \right),
\qquad \text{for all}~z \in U,~t \in [0,T], ~\text{and}~ 0 < \gamma < \delta,
\label{eq:expStabEquivProof58}
\end{equation}
where we have also used $p_2^{-1}(t) \ge t$.

\myStep{6}{Final arguments.}
Choose any $x \in B_\delta(\bO) \setminus \{ \bO \}$ with $x_1 \le 0$ and let $\gamma = \| x \|$.
Then $z = \frac{x}{\gamma} \in U$ and so by \eqref{eq:expStabEquivProof58}
\begin{equation}
\left\| \psi_t(x) \right\| \le \left( \frac{1}{4} + \alpha_0 \re^{\frac{-\beta_0 t}{M+1}} \right) \| x \|,
\qquad \text{for all}~ t \in [0,T].
\label{eq:expStabEquivProof60}
\end{equation}
In particular this gives
\begin{align}
\left\| \psi_t(x) \right\| &\le 2 \alpha_0 \| x \|,
\qquad \text{for all}~ t \in [0,T],~\text{and} \label{eq:expStabEquivProof61} \\
\left\| \psi_T(x) \right\| &\le \frac{1}{2} \| x \|, \label{eq:expStabEquivProof62}
\end{align}
where we have used \eqref{eq:TProof}.
By recursively applying \eqref{eq:expStabEquivProof61}--\eqref{eq:expStabEquivProof62}
to $x = \psi_{k T}(x)$ for all integers $k \ge 1$, we obtain
\begin{equation}
\| \psi_t(x) \| \le \alpha \re^{-\beta t} \| x \|, \qquad \text{for all}~ x \in B_\delta(\bO) ~\text{with}~ x_1 \le 0 ~\text{and all}~ t \ge 0,
\label{eq:expStabEquivProof63}
\end{equation}
with $\alpha = 4 \alpha_0$ and $\beta = \frac{\ln(2)}{T}$.

For $x \in B_\delta(\bO)$ but with instead $x_1 > 0$, the Filippov solution $\psi_t(x)$ quickly arrives at $\Sigma$
because $g^R(x) \approx \tilde{c}$ where $\tilde{c}_1 < 0$.
It is a simple exercise to show that \eqref{eq:expStabEquivProof63} also holds for all $x \in B_\delta(\bO)$ with $x_1 > 0$
by using a slightly larger value of $\alpha$.
\end{proof}

\section{Chaotic convergence}
\label{sec:chaos}
\setcounter{equation}{0}

In this section we study the system
\begin{equation}
\dot{x} = \begin{cases}
\begin{bmatrix}
-0.1 & 1 & 0 & 0 \\
-9 & 0 & 1 & 0 \\
-4 & 0 & 0 & 1 \\
-0.4 & 0 & 0 & 0
\end{bmatrix} x, & x_1 \le 0, \\
\begin{bmatrix} -1 \\ 0.4 \\ -0.2 \\ -0.04 \end{bmatrix}, & x_1 \ge 0.
\end{cases}
\label{eq:4d}
\end{equation}
This is a four-dimensional instance of \eqref{eq:F} in the normal form of \cite{Si18d}.
The numbers in \eqref{eq:4d} have been chosen to provide a succinct example of `chaotic' convergence to $x = \bO$.
This highlights the potential difficulty in determining whether or not $\bO$
is stable for a given system of the form \eqref{eq:F}.

By Theorem \ref{th:FilippovUniqueness}, \eqref{eq:4d} induces a semi-flow $\psi_t(x)$
for all $x \in \mathbb{R}^4$ and $t \ge 0$.
The semi-flow involves sliding motion on the attracting sliding region $\left\{ x \in \Sigma \,\big|\, x_2 > 0 \right\}$.
Orbits slide on this region until reaching the two-dimensional tangency surface
$\Sigma_{\rm tang} = \left\{ x \in \Sigma \,\big|\, x_2 = 0 \right\}$ at a point with $x_3 < 0$
from which they undergo regular motion in $\Omega_L$ until returning to the attracting sliding region and the process repeats, see Fig.~\ref{fig:ex4d2c}.
The dynamics of \eqref{eq:4d} can therefore be characterised by the induced return map on $\Sigma_{\rm tang}$.
But we can obtain a simpler description of the dynamics by also
using the time-scaled linear homogeneity property of \eqref{eq:4d}.

\begin{figure}[b!]
\begin{center}
\includegraphics[width=12cm]{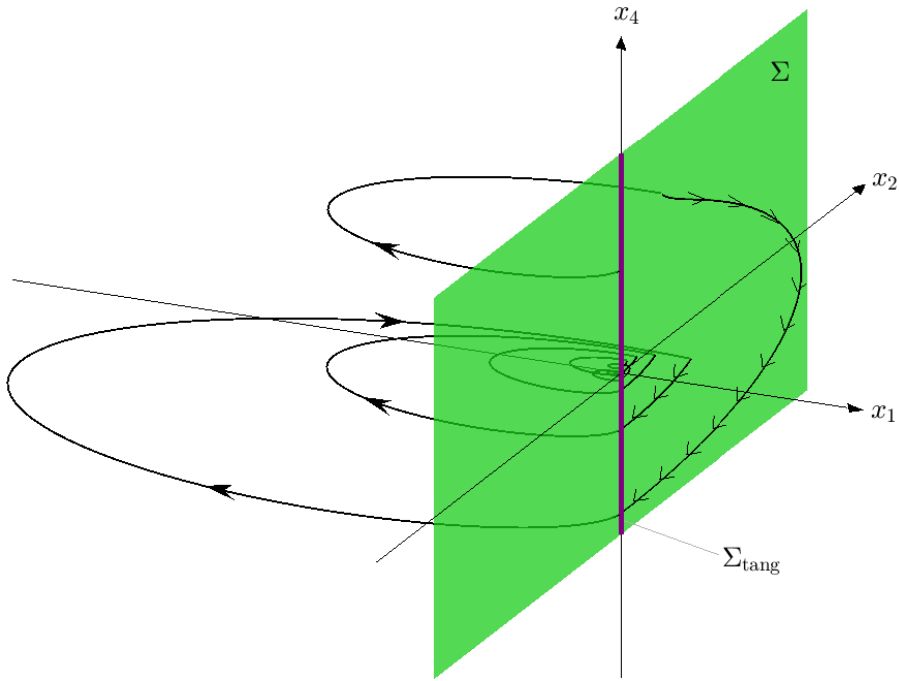}
\caption{
A typical Filippov solution of \eqref{eq:4d}.
The solution slides on $\Sigma$ until reaching $\Sigma_{\rm tang}$.
Note that this figure shows only three of the four variables.
\label{fig:ex4d2c}
} 
\end{center}
\end{figure}

\begin{figure}[b!]
\begin{center}
\includegraphics[width=12cm]{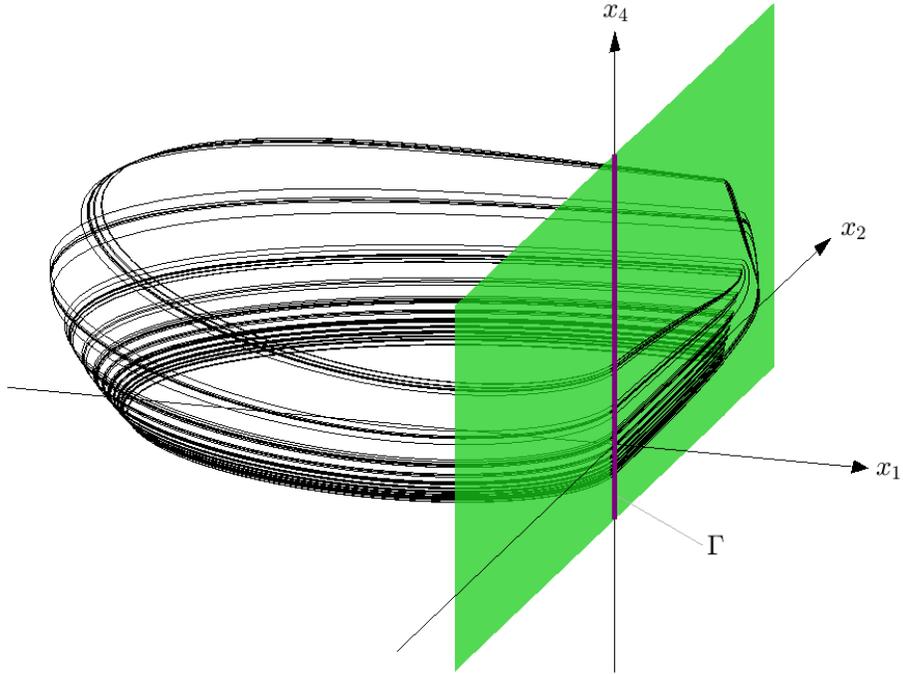}
\caption{
The Filippov solution of Fig.~\ref{fig:ex4d2c} projected onto the unit sphere $\mathbb{S}^3$.
The projected solution repeatedly intersects the one-dimensional manifold $\Gamma$ \eqref{eq:Gamma}
which is used to define the one-dimensional return map shown in Fig.~\ref{fig:ex4d2b}.
\label{fig:ex4d2a}
} 
\end{center}
\end{figure}

\begin{figure}[b!]
\begin{center}
\includegraphics[width=6cm]{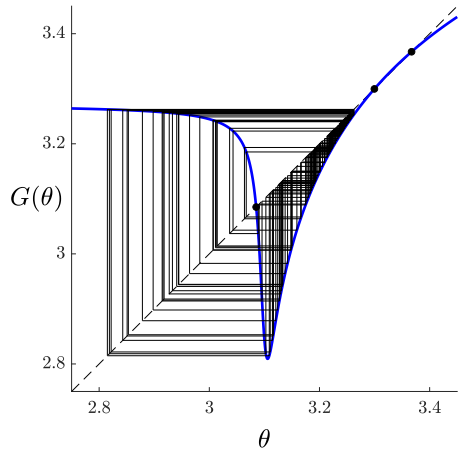}
\caption{
The return map for Filippov solutions of \eqref{eq:4d} projected onto $\mathbb{S}^3$
using the one-dimensional manifold $\Gamma$ as the domain of the map.
The map has three fixed points (black circles) that correspond to periodic orbits of the projected dynamics
(these correspond to Filippov solutions of \eqref{eq:4d} that spiral into the origin in a simple fashion).
We also show, as a cobweb diagram, the orbit of $G$ corresponding to the solution shown in Fig.~\ref{fig:ex4d2a}.
\label{fig:ex4d2b} 
} 
\end{center}
\end{figure}

Let $r_t(x) = \| \psi_t(x) \|$ using the Euclidean norm.
Then $\chi_t(x) = \frac{\psi_t(x)}{r_t(x)}$ is the projection of $\psi_t(x)$ onto the unit sphere,
$\mathbb{S}^3 = \{ x \,|\, x_1^2 + x_2^2 + x_3^2 + x_4^2 = 1 \}$.
Fig.~\ref{fig:ex4d2a} shows a typical projected orbit $\chi_t(x)$.
Projected orbits repeatedly intersect the projection of $\Sigma_{\rm tang}$, with $x_3 < 0$, onto $\mathbb{S}^3$.
This is the one-dimensional manifold
\begin{equation}
\Gamma = \left\{ \zeta(\theta) \,\big|\, \theta \in \left( \tfrac{\pi}{2}, \tfrac{3 \pi}{2} \right) \right\},
\label{eq:Gamma}
\nonumber
\end{equation}
where $\zeta : \left( \tfrac{\pi}{2}, \tfrac{3 \pi}{2} \right) \to \mathbb{S}^3$ is the function
\begin{equation}
\zeta(\theta) = \big( 0,0,\cos(\theta),\sin(\theta) \big).
\nonumber
\end{equation}
Fig.~\ref{fig:ex4d2b} shows the induced return map on $\Gamma$, call it $G(\theta)$.
The map $G$ is unimodal over the range of $\theta$-values shown
and appears to have a chaotic attractor corresponding the orbit shown in Fig.~\ref{fig:ex4d2a}
(from numerical simulations we estimate its Lyapunov exponent to be about $0.22$).

To clarify, given $\theta \in \left( \tfrac{\pi}{2}, \tfrac{3 \pi}{2} \right)$ the value $G(\theta)$
is such that $\zeta(G(\theta))$ is the next intersection of $\chi_t(\zeta(\theta))$ with $\Gamma$.
We write $\zeta(G(\theta)) = \chi_T(\zeta(\theta))$ where $T = T(\theta) > 0$ is the corresponding evolution time.
To then characterise the amount by which the corresponding non-projected orbit
heads towards or away from $\bO$, 
we define $D : \left( \tfrac{\pi}{2}, \tfrac{3 \pi}{2} \right) \to \mathbb{R}$ by $D(\theta) = r_T(\zeta(\theta))$.
In view of the time-scaled linear homogeneity property of \eqref{eq:4d},
the function $D$ gives the change in norm as an orbit of \eqref{eq:4d}
undergoes one excursion from any point on $\Sigma_{\rm tang}$ back to $\Sigma_{\rm tang}$.

Numerically we found that the average value of $D$ over the projected orbit $\chi_t(x)$
shown in Fig.~\ref{fig:ex4d2a} is about $0.54$.
Since this value is less than $1$, the corresponding orbit $\psi_t(x)$ of \eqref{eq:4d}
converges to the origin as $t \to \infty$ (as evident in Fig.~\ref{fig:ex4d2c}).
However, this does not imply $\bO$ is asymptotically stable.
The map $G$ has many invariant sets (presumably an infinity of periodic solutions dense in some open subset of $\left( \tfrac{\pi}{2}, \tfrac{3 \pi}{2} \right)$).
If the average value of $D$ is greater than $1$ for any of these then $\bO$ is unstable.
As with piecewise-linear maps \cite{Si20d},
for a system similar to \eqref{eq:4d} it is presumably possible for $\bO$ to be unstable even if almost all Filippov solutions converge to $\bO$.
In this situation it is not clear what computational method would effectively establish that $\bO$ is indeed unstable.

\section{Discussion}
\label{sec:conc}
\setcounter{equation}{0}

The results of this paper are motivated by a desire to determine
whether or not a boundary equilibrium of a given Filippov system $f$ is stable.
Together Theorems \ref{th:expStabEquiv} and \ref{th:asyStabImpliesExpStab}
tell us that, assuming $c_1 < 0$, if $x^*$ is asymptotically stable for the reduced system $F$
then it is also asymptotically stable for the original system $f$.
But what if $x^*$ is not asymptotically stable for $F$?
Does this imply $x^*$ is not asymptotically stable for $f$?
The answer is no (a counterexample is readily constructed by letting $A$ be the zero matrix in \eqref{eq:F}).
However, this may be true if asymptotic stability is replaced by Lyapunov stability:

\begin{conjecture}
If the equilibrium $x = \bO$ is unstable (i.e.~not Lyapunov stable) for $F$ \eqref{eq:F},
then it is also unstable for $f$ \eqref{eq:f}.
\label{cj:Lyap}
\end{conjecture}

As with piecewise-linear maps \cite{Si20d},
Conjecture \ref{cj:Lyap} stems from the observation that if $F$ has an orbit emanating from the origin
then we expect $f$ to have an analogous orbit emanating from the origin at the same asymptotic rate.
This essentially claims that an unstable manifold in $F$ is also present in $f$.
By combining Theorems \ref{th:expStabEquiv} and \ref{th:asyStabImpliesExpStab} and Conjecture \ref{cj:Lyap}
we obtain Fig.~\ref{fig:stabilityImplicationChart} that summarises the implications between
the different types of stability for $f$ and $F$.

\begin{figure}[b!]
\begin{center}
\setlength{\unitlength}{1cm}
\begin{picture}(10.8,8.1)
\put(0,0){\includegraphics[height=8.1cm]{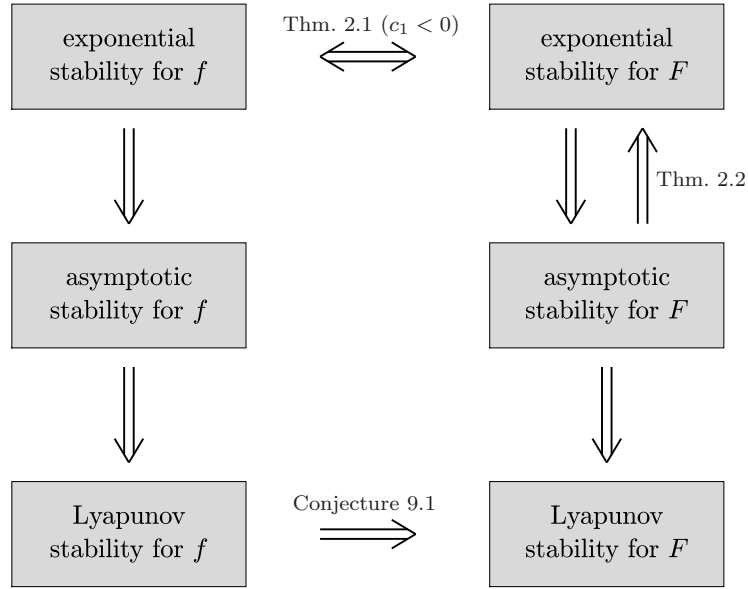}}
\put(4.4,1.2){\scriptsize Conjecture \ref{cj:Lyap}}
\put(4.27,7.57){\scriptsize Thm.~\ref{th:expStabEquiv} ($c_1 < 0$)}
\put(9.23,5.5){\scriptsize Thm.~\ref{th:asyStabImpliesExpStab}}
\end{picture}
\caption{
Implications between the three types of stability listed in Definition \ref{df:stable}
for a boundary equilibrium of a Filippov system $f$ of the form \eqref{eq:f}
and its corresponding local approximation $F$ \eqref{eq:F}.
For any system exponential stability implies asymptotic stability
and asymptotic stability implies Lyapunov stability.
Theorems \ref{th:expStabEquiv} and \ref{th:asyStabImpliesExpStab} and Conjecture \ref{cj:Lyap} (if true)
provide additional implications as shown.
Adjoining implications can be composed,
for example to see that asymptotic stability for $F$ implies asymptotic stability for $f$
(but the converse is not necessarily true).
\label{fig:stabilityImplicationChart}
}
\end{center}
\end{figure}

As a final comment, almost all $n$-dimensional systems of the form \eqref{eq:F} with $c_1 < 0$ can be reduced, via an affine coordinate change,
to the normal form of \cite{Si18d} that involves $2 n - 1$ parameters.
In the case $n=3$, a numerically tractable and likely insightful project would be
to describe the subset of five-dimensional parameter space within which $\bO$ is stable.
In the case $n=4$, it would be useful to understand how the stability changes as the numbers in \eqref{eq:4d} are varied.
However, as mentioned above, for this system it is not clear what numerical method would most effectively characterise stability.

\appendix

\section{Proof of Lemma \ref{le:uniformConvergence}}
\label{app:uniformConvergence}
\setcounter{equation}{0}

\begin{proof}
By asymptotic stability there exists $\delta > 0$ such that $\varphi_t(x) \to \bO$ as $t \to \infty$ for all $x \in \overline{B}_\delta(\bO)$.
Choose any $\ee > 0$.
Since $\bO$ is Lyapunov stable, there exists $\delta_1 > 0$ such that
\begin{equation}
\varphi_t(x) \in B_\ee(\bO),
\quad \text{for all~} x \in B_{\delta_1}(\bO) \text{~and all~} t \ge 0.
\label{eq:uniformConvergenceProof1}
\end{equation}
For any $x \in \overline{B}_\delta(\bO)$ there exists $\hat{T}(x) \ge 0$ such that
\begin{equation}
\varphi_{\hat{T}(x)}(x) \in B_{\frac{\delta_1}{2}}(\bO).
\label{eq:uniformConvergenceProof2}
\end{equation}
Since $\varphi_t(x)$ is a continuous function there exists $\eta(x) > 0$ such that
\begin{equation}
\big\| \varphi_{\hat{T}(x)}(x) - \varphi_{\hat{T}(x)}(y) \big\| < \frac{\delta_1}{2},
\quad \text{for all~} y \in B_{\eta(x)}(x) \cap \overline{B}_\delta(\bO).
\label{eq:uniformConvergenceProof3}
\end{equation}
The balls $B_{\eta(x)}(x)$ cover the compact set $\overline{B}_\delta(\bO)$,
thus there exists a subcover using points $x_1,\ldots,x_m \in \overline{B}_\delta(\bO)$.
Let $T = \max \left[ \hat{T}(x_1),\ldots,\hat{T}(x_m) \right]$ and choose any $x \in \overline{B}_\delta(\bO)$.
Let $i \in \{ 1,\ldots,m \}$ be such that $x \in B_{\eta(x_i)}(x_i)$.
Then
\begin{equation}
\big\| \varphi_{\hat{T}(x_i)}(x) \big\|
\le \big\| \varphi_{\hat{T}(x_i)}(x) - \varphi_{\hat{T}(x_i)}(x_i) \big\|
+ \big\| \varphi_{\hat{T}(x_i)}(x_i) \big\|
< \frac{\delta_1}{2} + \frac{\delta_1}{2} = \delta_1 \,,
\nonumber
\end{equation}
where we have used \eqref{eq:uniformConvergenceProof2} and \eqref{eq:uniformConvergenceProof3}.
By \eqref{eq:uniformConvergenceProof1} and the group property of the semi-flow \eqref{eq:semiFlowGroupProperty} we have
$\varphi_t(x) = \varphi_{t - \hat{T}(x_i)} \big( \varphi_{\hat{T}(x_i)}(x) \big) \in B_\ee(\bO)$
for all $t \ge T$.
\end{proof}


\end{document}